\documentclass[preprint, 12pt, english]{elsarticle}
\usepackage{amsmath}
\usepackage{latexsym, amssymb}
\usepackage{amsthm}
\usepackage[small,nohug,heads=littlevee]{diagrams} 

\newtheorem{thm}{Theorem}[section] 

\newtheorem{cor}[thm]{Corollary}

\newtheorem{exmpl}[thm]{Example}

\newtheorem{prop}[thm]{Proposition}

\newtheorem{rem}[thm]{Remark}

\newcommand\operA[2]{{\if!#2!\operatorname{#1}\else{\operatorname{#1}_{#2}^{\phantom{I}}}\fi}} 

\newcommand\Cref[1]{{Corollary~\ref{#1}}}%
\newcommand{\Trace}[1][]{\if!#1!\operatorname{Tr}\else{\operatorname{Tr}_{#1}^{\phantom{I}}}\fi} 

\long\def\forget#1\forgotten{{}} %

\def\({\left(}
\def\){\right)}

\newif\iffurther
\furtherfalse

\newif\ifXY 
\XYtrue     
\ifXY

\input xy
\input xyidioms.tex
\usepackage{xy}
\xyoption{all} %
\fi 

\usepackage{babel}

\journal{??}

\begin{document}

\begin{frontmatter}

\title{Standard Polynomial Equations over Division Algebras}
\author{Adam Chapman}
\ead{adam1chapman@yahoo.com}
\author{Casey Machen}
\ead{machenc@gmail.com}
\address{Department of Mathematics, Michigan State University, East Lansing, MI 48824}

\begin{abstract}
Given a central division algebra $D$ of degree $d$ over a field $F$, we associate to any standard polynomial $\phi(z)=z^n+c_{n-1} z^{n-1}+\dots+c_0$ over $D$ a ``companion polynomial" $\Phi(z)$ of degree $nd$ with coefficients in $F$. The roots of $\Phi(z)$ in $D$ are exactly the set of conjugacy classes of the roots of $\phi(z)$. When $D$ is a quaternion algebra, we explain how all the roots of $\phi(z)$ can be recovered from the roots of $\Phi(z)$. Along the way, we are able to generalize a few known facts from $\mathbb{H}$ to any division algebra. The first is the connection between the right eigenvalues of a matrix and the roots of its characteristic polynomial. The second is the connection between the roots of a standard polynomial and left eigenvalues of the companion matrix.
\end{abstract}


\end{frontmatter}

\section{Introduction}

Let $F$ be a field. A \emph{central division algebra} is a division algebra $D$ which is of finite dimension over its center $F$. The dimension $[D:F]$ is a square, and the integer $d=\sqrt{[D:F]}$ is the \emph{degree} of $D$ over $F$.

Throughout this paper, we consider \emph{standard polynomials} over $D$. These are monic polynomials with coefficients appearing only on the left-hand side of the variable:
$$
\phi(z)=z^n+c_{n-1} z^{n-1}+\dots+c_0
$$
where $c_i\in D$. 
By a \emph{root} of a standard polynomial $\phi$ over $D$ we mean an element $\lambda \in D$ satisfying $\phi(\lambda)=0$.
We are interested in finding these roots. 

One of the earliest papers to explore standard polynomials is \cite{Niven1941}. Niven proves that when $D$ is a real quaternion algebra and $F$ is a real closed field, every standard polynomial has a root in $D$. A good survey on polynomial equations over more general division rings can be found in \cite{GordonMotzkin1965}.
 
The known results when it comes to polynomial equations over quaternion algebras can be outlined as follows: In \cite{HuangSo2002} a formula was given for the roots of any quadratic (i.e. $n=2$) standard polynomial over $\mathbb{H}$. In \cite{Abrate2009} this formula was generalized to any quaternion algebra over fields of characteristic not 2, and in \cite{Chapman2015} to quaternion algebras over fields of characteristic 2 as well. In \cite{SerodioPereiraVitoria2001} it was shown that the roots of any standard polynomial of degree $n$ over $\mathbb{H}$ are also roots of the ``companion polynomial".
This polynomial is the characteristic polynomial of the companion matrix, which is a polynomial of degree $2n$ with coefficients in $\mathbb{R}$. In \cite{JanovskaOpfer2010} it was shown how the roots of the original polynomial can be recovered from the roots of the companion polynomial.

In this paper we generalize the result from \cite{SerodioPereiraVitoria2001} to any central division algebra. In the special case of quaternion algebras, it provides a method for finding all the roots of a given standard polynomial. We make use of the theory of left and right eigenvalues, which was studied extensively for matrices over $\mathbb{H}$ and here we present some generalizations to any central division algebra.

\section{Right Eigenvalues of Matrices over Central Division Algebras}

Let $D$ be a central division algebra of degree $d$ over $F$ and let $n$ be a positive integer.
We write $M_n(D)$ for the ring of $n \times n$ matrices over $D$, and $D^n$ for the column vector space of dimension $n$ over $D$.
The vector space $D^n$ is both a left $M_n(D)$-module and a $D$ bi-module.

A right eigenvalue of a matrix $A$ in $M_n(D)$ is an element $\lambda \in D$ for which there exists a vector $v \in D^n$ satisfying $A v=v \lambda$. A left eigenvalue is defined in a similar way, just with $A v=\lambda v$.

The ring $M_n(D)$ can be identified with the ring of endomorphisms of $D^n$ as a right $D$-module.
Let $K$ be a maximal subfield of $D$.
The ring $D$ can be viewed as a vector space over $K$, and in particular $[D:K]=d$.
Therefore, $D^n$ (as a right $D$-module) can be identified with the vector space $K^{n d}$.
Let $g : D^n \rightarrow K^{n d}$ be a $K$-vector space isomorphism.
Note that $g(v \lambda)=g(v) \lambda$ for any $v \in D^n$ and $\lambda \in K$.
Via this identification, each element in $M_n(D)$ can be identified with a $K$-linear endomorphism of $K^{n d}$.
This gives rise to a $K$-linear map
$f : M_n(D) \rightarrow M_{n d}(K)$, which is an embedding.
Note that for any $A \in M_n(D)$ and $v \in D^n$, we have $g(A v)=f(A) g(v)$. This can be expressed in the following commutative diagram:
$$
\begin{diagram}
D^n		&\rTo^{g}		&K^{nd}	\\
\dTo^{\_ \cdot A}	&			&\dTo_{\_ \cdot f(A)} \\
D^n		&\rTo^{g}		&K^{nd}.
\end{diagram}
$$

For each $A\in M_n(D)$, the \emph{characteristic polynomial} of $A$ is
$$
\Phi_A(z) := \det\big(f(A)-zI\big)
$$
where $I$ is the identity matrix in $M_{n d}(K)$. The coefficients of $\Phi_A(z)$ are known to lie in $F$. Furthermore, they are independent of both the choice of maximal subfield $K$ of $D$ as well as choice of isomorphism $g:D^n\to K^{nd}$ (see for example \cite[Section 4.5]{GilleSzamuely2006}).

We say that two elements $d_1$ and $d_2$ in $D$ are \emph{conjugate} if there exists some nonzero $q \in D$ such that $d_1=q d_2 q^{-1}$. The conjugacy class of an element in $D$ is the set of all its conjugates.

It was shown in \cite{Lee1949} that the conjugacy classes of right eigenvalues of an $n \times n$ matrix $A$ over $\mathbb{H}=\mathbb{R}+i \mathbb{R}+j \mathbb{R}+i j \mathbb{R}$ are the conjugacy classes of roots of the complex roots of the characteristic polynomial of
$\left(\begin{array}{rr}
A_1 & -\overline{A_2}\\
A_2 & \overline{A_1}
\end{array}\right)$
where $A_1$ and $A_2$ are the $n \times n$ matrices over $\mathbb{R}(i)$($\cong \mathbb{C}$) satisfying $A=A_1+j A_2$.
We now generalize this fact to any central division algebra.

\begin{thm}\label{right}
Let $D$ be a central division $F$-algebra of degree $d$, and $n$ be a positive integer.
Let $A \in M_n(D)$ and $\lambda \in D$. Then $\lambda$ is a right eigenvalue of $A$ if and only if $\Phi_A(\lambda)=0$.
\end{thm}

\begin{proof}
Let $K$ be a maximal subfield of $D$ containing $\lambda$.
Let $g : D^n \rightarrow K^{n d}$ and $f : M_n(D) \rightarrow M_{n d}(K)$ be as defined above.

Assume $\lambda$ is a right eigenvalue of $A$.
Then there exists a nonzero vector $v \in D^n$ such that $A v=v \lambda$. Then $f(A) g(v)=g(A v)=g(v \lambda)=g(v) \lambda$. The element $\lambda$ is an eigenvalue of $f(A)$ (in the classical sense), and so a root of $\Phi_A(z)$.

Assume $\Phi_A(\lambda)=0$.
Then $\lambda$ is an eigenvalue of $f(A)$, so there exists a nonzero vector $w \in K^{n d}$ such that $f(A) w=w \lambda$.
Let $v=g^{-1}(w)$. Then $A v=v \lambda$, hence $\lambda$ is a right eigenvalue of $A$.
\end{proof}

Note that since $\Phi_A(z)$ has coefficients in $F$, for each root $\lambda\in D$ of $\Phi_A(z)$, all of its conjugates are roots as well. So one can consider the roots of $\Phi_A(z)$ as a collection of conjugacy classes of elements of $D$.
Each such conjugacy class corresponds to the isomorphism class of a finite field extension $F(\lambda)/F$, which can be identified with a subfield of a fixed algebraic closure $\overline{F}$ of $F$.
Therefore, in order to find all the roots of $\Phi_A(z)$ that lie in $D$, one can solve it as a polynomial over $\overline{F}$, and then for each root $\lambda\in\overline{F}$ check whether $F(\lambda)$ is a subfield of $D$. If so, $\lambda$ is a root of $\Phi_A(z)$ in $D$, and otherwise it is not.

\section{Standard Polynomials and Left Eigenvalues of the Companion Matrix}

Let $\phi(z)=z^n+c_{n-1} z^{n-1}+\dots+c_0$ be a standard polynomial with coefficients $c_0,\dots,c_{n-1}$ in $D$. We want to find the roots of $\phi(z)$ in $D$, i.e. all the elements $\lambda \in D$ satisfying $\phi(\lambda)=0$.

Given such a polynomial, we define its companion matrix to be
$$C_\phi=\left(\begin{array}{rrrrr}
0 & 1 & 0 & \dots & 0\\
0 & 0 & 1 & & 0\\
\vdots & & & \ddots & \\
0 &0 & \dots & 0& 1\\
-c_0 & -c_1 & \dots &-c_{n-2} & -c_{n-1}
\end{array}\right).$$

\begin{thm}\label{Comp}
The roots of $\phi(z)$ are exactly the left eigenvalues of $C_\phi$. Furthermore, given such a left eigenvalue $\lambda$ of the companion matrix $C_\phi$, the vector $v=\left(\begin{array}{r} 1 \\ \lambda \\ \vdots\\ \lambda^{n-1} \end{array}\right)$ is a corresponding left eigenvector, i.e. a vector satisfying $C_\phi v=\lambda v$.
\end{thm}

\begin{proof}
The element $\lambda \in D$ is a left eigenvalue of $C_\phi$ if and only if there exists a nonzero vector 
$$v=\left(\begin{array}{r}
v_1\\
\vdots\\
v_n
\end{array}\right) \in D^n
$$ satisfying $C_\phi v=\lambda v$.
This equality is equivalent to the system
\begin{eqnarray*}
v_2 & = & \lambda v_1\\
\vdots\\
v_n & = & \lambda v_{n-1}\\
-c_0 v_1-\dots-c_{n-1} v_n & = & \lambda v_n
\end{eqnarray*}
Note that since $v\ne0$, $v_1\ne0$.
The first $n-1$ equations mean that $v$ is $\left(\begin{array}{r} 1 \\ \lambda \\ \vdots\\ \lambda^{n-1} \end{array}\right)v_1$ and the last equation then becomes 
$$(c_0+c_1 \lambda+\dots+c_{n-1} \lambda^{n-1}+\lambda^n) v_1=0.$$
By dividing by $v_1$ from the right, we obtain that $\lambda$ is a root of $\phi(z)$.

In the other direction, it is straight-forward to see that for any $\lambda \in D$ satisfying $c_0+c_1 \lambda+\dots+c_{n-1} \lambda^{n-1}+\lambda^n=0$, 
$$C_\phi \left(\begin{array}{r} 1 \\ \lambda \\ \vdots\\ \lambda^{n-1} \end{array}\right)=\lambda \left(\begin{array}{r} 1 \\ \lambda \\ \vdots\\ \lambda^{n-1} \end{array}\right)$$
which means that $\lambda$ is a left eigenvalue of $C_\phi$.
\end{proof}

\begin{rem}
Similar connections between the polynomial and its companion matrix were also pointed out in \cite{LamLeroyOzturk2008}.
The fact that every root of $\phi(z)$ is also a left eigenvalue of $C_\phi$ can also be obtained as a result of Lemma 4.7 in that paper.
\end{rem}

\begin{cor}\label{leftright}
Every left eigenvalue of $C_\phi$ is also a right eigenvalue. 
\end{cor}

\begin{proof}
Let $\lambda$ be a left eigenvector. Then $\left(\begin{array}{r} 1 \\ \lambda \\ \vdots\\ \lambda^{n-1} \end{array}\right)$ is the corresponding eigenvector. Now
$$\left(\begin{array}{r} 1 \\ \lambda \\ \vdots\\ \lambda^{n-1} \end{array}\right) \lambda=\lambda \left(\begin{array}{r} 1 \\ \lambda \\ \vdots\\ \lambda^{n-1} \end{array}\right).$$
\end{proof}

\begin{cor}\label{vice}
Every right eigenvalue of $C_\phi$ is conjugate to some left eigenvalue of $C_\phi$.
\end{cor}

\begin{proof}
Assume $\lambda$ is a right eigenvalue.
Then there exists a nonzero vector 
$$v=\left(\begin{array}{r}
v_1\\
\vdots\\
v_n
\end{array}\right) \in D^n
$$ satisfying $C_\phi v=v \lambda$.
From this equality we obtain the system
\begin{eqnarray*}
v_2 & = & v_1 \lambda\\
\vdots\\
v_n & = & v_{n-1} \lambda\\
-c_0 v_1-\dots-c_{n-1} v_n & = &v_n \lambda
\end{eqnarray*}
Substituting the first $n-1$ equations in the last one we obtain
$$c_0 v_1+c_1 v_1 \lambda+\dots+c_{n-1} v_1 \lambda^{n-1}+v_1 \lambda^n=0.$$
Recall that since $v$ is nonzero, $v_1$ is nonzero.
Take $\lambda'=v_1 \lambda v_1^{-1}$.
Then $\lambda'$ is a root of $\phi(z)$ and so a left eigenvalue of $C_\phi$
\end{proof}

Let the characteristic polynomial of $C_\phi$ be denoted by $\Phi(z)$.
We call it the \emph{companion polynomial} of $\phi(z)$.
Note that it is of degree $n d$ and has coefficients in $F$.

\begin{rem}
For a quaternionic (i.e. degree 2) division algebra, it is not mentioned explicitly in \cite{JanovskaOpfer2010}, but the companion polynomial
$q(z)=z^{2 n}+(\overline{c_{n-1}}+c_{n-1}) z^{2n-1}+\dots$ associated to any polynomial $p(z)=z^n+c_{n-1} z^{n-1}+\dots$ is indeed the characteristic polynomial of the companion matrix.
This can be easily verified by a straight-forward computation: take the companion matrix $C_p$, write it as $A_1+j A_2$ for the appropriate $A_1,A_2 \in M_n(\mathbb{R}(i))$ and compute the characteristic polynomial of $$\left(\begin{array}{rr}
A_1 & -\overline{A_2}\\
A_2 & \overline{A_1}
\end{array}\right).$$
\end{rem}

\begin{thm}\label{inc}
The roots of $\phi(z)$ are also roots of $\Phi(z)$, and each conjugacy class of roots of $\Phi(z)$ contains a root of $\phi(z)$.
\end{thm}

\begin{proof}
By Theorem \ref{Comp} the roots of $\phi(z)$ are left eigenvalues of the companion matrix.
By Corollary \ref{leftright} those left eigenvalues are also right eigenvalues, and by Theorem \ref{right} the right eigenvalues are roots of the companion polynomial.
The second assertion is immediate from Corollary \ref{vice}.
\end{proof}

\section{Standard Polynomials over Quaternion Algebras}

Let $\phi(z)$ be again a standard polynomial over a central division algebra $D$ of degree $d$ over $F$. The algebra $D$ is called a \emph{quaternion algebra} if $d=2$. We begin this section with a general discussion, and we conclude by explaining how the roots of $\phi(z)$ can be recovered from the roots of its companion polynomial $\Phi(z)$ in the case that $D$ is a quaternion algebra.

Every element $\lambda \in D$ has a characteristic polynomial,
which means that 
$\lambda^d+b_{d-1} \lambda^{d-1}+\dots+b_0=0$
for some $b_{d-1},\dots,b_0 \in F$.
Note that the characteristic polynomial of $\lambda$ does not change if we conjugate $\lambda$ by any $q \in D^\times$.
Write $C_{d-1}(z),\dots,C_0(z)$ for the functions from $D$ to $F$ such that
$z^d+C_{d-1}(z) z^{d-1}+\dots+C_0(z)$ is the characteristic polynomial of $z$. We may write $\phi(z)$ as $\psi_{d-1}(z) z^{d-1}+\dots+\psi_0(z)$ using the functions $C_{d-1}(z),\dots,C_0(z)$.

\begin{exmpl} If $\phi(z)=z^3$ and $d=2$ then 
\begin{eqnarray*}
\phi(z) &= &z^2 \cdot z+c=(-C_1(z) z-C_0(z)) z+c=-C_1(z) z^2-C_0(z) z\\&=&-C_1(z) (-C_1(z) z-C_0(z))-C_0(z) z \\&=&(C_1(z)^2-C_0(z))z+C_1(z) C_0(z)
\end{eqnarray*}
 and so $\psi_1(z)=C_1(z)^2-C_0(z)$ and $\psi_0(z)=C_1(z) C_0(z)$.
 \end{exmpl}

By Theorem \ref{inc}, the roots of $\phi(z)$ are also roots of $\Phi(z)$.
Assuming we have all the roots of $\Phi(z)$ on hand, we want to recover the roots of $\phi(z)$.
The following statement is immediate:

\begin{prop}
Given a root $\lambda$ of $\Phi(z)$, $\lambda$ is a root of $\phi(z)$ if and only if $\lambda$ is a root of the polynomial $\Psi(z)=\psi_{d-1}(\lambda) z^{d-1}+\dots+\psi_0(\lambda)$.
\end{prop}

In the case $d=2$, the polynomial $\Psi(z)$ defined in the proposition above is either a linear polynomial or a constant.
It cannot be a nonzero constant, because that would mean that no element in the conjugacy class of $\lambda$ is a root of $\phi$, contradictory to Theorem \ref{inc}. We now have a way of finding roots of $\phi(z)$ from the roots of its companion polynomial $\Phi(z)$ when $D$ is a quaternion algebra:

\begin{cor}
Assume $d=2$ and fix a root $\lambda$ of $\Phi(z)$. If $\Psi(z)$ is constantly zero then every element in the conjugacy class of $\lambda$ is a root of $\phi(z)$.
If $\Psi(z)$ is a linear polynomial, then the root of $\phi(z)$ in the conjugacy class of $\lambda$ can be found simply by solving the linear equation $\Psi(z)=0$.
\end{cor}

\section*{Acknowledgements}

We thank Rajesh Kulkarni for his comments on the manuscript and Uzi Vishne and Uriya First for the useful discussions.

\section*{Bibliography}
\bibliographystyle{amsalpha}
\bibliography{bibfile}
\end{document}